\documentclass[12pt]{amsart}
\usepackage{cases}
\usepackage{amsthm}
\usepackage{amsmath}
\usepackage{amscd}
\usepackage{amssymb}
\usepackage{graphicx}
\usepackage[mathscr]{eucal}
\usepackage[colorlinks,linkcolor=blue,citecolor=blue, pdfstartview=FitH]{hyperref}

\input xy
\xyoption{all} \numberwithin{equation}{section}
\begin{document}

\title[On Hecke lifting conjecture for framed knots]
{On Hecke lifting conjecture for framed knots}
\author[Shengmao Zhu]{Shengmao Zhu}
\address{Department of mathematics \\
Zhejiang Normal University  \\
Jinhua,  Zhejiang {\rm 321004}, China }
\email{szhu@zju.edu.cn}

\begin{abstract}
Motivated by an amazing integrality structure conjecture for the $U(N)$ Chern-Simons quantum invariants of framed knots investigated by Mari\~no and Vafa, a new conjectural formula, named Hecke lifting conjecture, was proposed in \cite{CLPZ23} for framed links.  This note is devoted to the study of this Hecke lifting conjecture. We prove this conjecture for torus knots using the explicit formulas of colored HOMFLY-PT invariants of torus knots, and we also verify the conjecture in a limit form for any framed knots. 
\end{abstract}

\maketitle

\theoremstyle{plain} \newtheorem{thm}{Theorem}[section] \newtheorem{theorem}[%
thm]{Theorem} \newtheorem{lemma}[thm]{Lemma} \newtheorem{corollary}[thm]{%
Corollary} \newtheorem{proposition}[thm]{Proposition} \newtheorem{conjecture}%
[thm]{Conjecture} \theoremstyle{definition}
\newtheorem{remark}[thm]{Remark}
\newtheorem{remarks}[thm]{Remarks} \newtheorem{definition}[thm]{Definition}
\newtheorem{example}[thm]{Example}






\section{Introduction}

The seminal work \cite{Witten} of E. Witten
showed that Chern-Simons gauge theory provides a natural way to
study the quantum invariants \cite{Jones}. In this framework, the expectation
value of Wilson loop along a link $\mathcal{L}$ in $S^3$ gives a
topological invariant of the link depending on the representation of
the gauge group. N. Reshetikhin and V. Turaev \cite{RT} gave a
mathematical construction of this link invariant by using the
representation theory of the quantum group. In particular, the gauge group $%
SU(N)$ with irreducible representation will give rise to the colored
HOMFLY-PT invariant of the link $\mathcal{L}$. In another
fundamental work of Witten \cite{Witten2}, the $U(N)$ Chern-Simons
gauge theory on a
three-manifold $M$ was interpreted as an open topological string theory on $%
T^*M$ with $N$ topological branes wrapping $M$ inside $T^*M$.
Furthermore, Gopakumar-Vafa \cite{GV} conjectured that the large $N$
limit of $SU(N)$ Chern-Simons gauge theory on $S^3$ is equivalent to
the closed topological string theory on the resolved conifold. This
highly nontrivial string
duality was first checked for the case of the unknot by Ooguri-Vafa \cite{OV}%
. Later, a series of work \cite{LMV,LM} based on the large $N$
Chern-Simons/topological string duality, conjectured an expansion of
the Chern-Simons partition functions in terms of an infinite
sequence of integer invariants, which is called the
Labastida-Mari\~no-Ooguri-Vafa (LMOV) conjecture. This integrality
conjecture serves as an essential evidence of the
Chern-Simons/topological string duality and was proved in
\cite{LP1}. When considering the framing dependence for $U(N)$
Chern-Simons gauge theory, the integrality structure is even more amazing as described in \cite
{MV}. This framing dependence integrality structure conjecture is called the framed Labastida-Mari\~no-Ooguri-Vafa (framed LMOV) conjecture in \cite{CLPZ23}. In the joint work \cite{CLPZ23} with Chen, Liu and Peng, we proposed Hecke lifting conjecture for framed links in the process of attacking the framed LMOV conjecture. We also proposed two congruence skein relations, which imply the Hecke lifting conjecture for framed links, and proved that these congruence skein relations hold in many different situations. The goal of this paper is to provide a direct approach to the Hecke lifting conjecture, and we hope to show some new insights on the framed LMOV conjecture. Now, let us  briefly describe the main results of this article.

\subsection{Hecke lifting conjecture for framed links}
The HOMFLY-PT polynomial is a two-variable link invariant which was
first discovered by Freyd-Yetter, Lickorish-Millet, Ocneanu, Hoste
and Przytycki-Traczyk.  Let $\mathcal{L}$ be
an oriented link in $S^3$, the framed HOMFLY-PT polynomial of $\mathcal{L}$ is a two-variable polynomial denoted by 
$\mathcal{H}(\mathcal{L};q,a)$ (cf. (\ref{formula-H}) for the definition).  
Suppose $\mathcal{L}$ has $L$ components $\mathcal{K}_\alpha$ ($%
\alpha=1,...,L$), and each component $\mathcal{K}_\alpha$ has framing $\tau^\alpha\in \mathbb{Z}$. Taking $\vec{\tau}=(\tau^1,..,\tau^L)$, we say that $\mathcal{L}$ has framing $\vec{\tau}$.  Let $\vec{\lambda}=(\lambda^1,...,\lambda^L)$,
where each $\lambda^\alpha$  $(\alpha=1,..,L)$ denotes a partition
of a positive integer, the framed colored HOMFLY-PT
invariant of $\mathcal{L}$ colored by $\vec{\lambda}$ is defined as
\begin{align}
 \mathcal{H}(\mathcal{L}\star
\otimes_{\alpha=1}^LQ_{\lambda^\alpha};q,a),
\end{align}
where $\mathcal{L}\star
\otimes_{\alpha=1}^LQ_{\lambda^\alpha}$ denotes the link obtained by  $\mathcal{L}$ decorated
by the
element $\otimes_{\alpha=1}^LQ_{\lambda^\alpha}$, where each $%
Q_{\lambda^\alpha}$ is in the skein of annulus $\mathcal{C}_+$ (cf. Section \ref{section-coloredH} for definitions).

For a
partition $\mu$, let $P_{\mu}=\sum_{\lambda}\chi_{%
\lambda}(\mu)Q_{\lambda}$, where $\chi_{\lambda}(\mu)$ is the value
of the
character $\chi_{\lambda}$ of the symmetric group corresponding to the conjugate class $%
C_{\mu} $. From the point of view of the HOMFLY-PT skein theory, the element
$P_{\mu}\in \mathcal{C}_+$ takes a simple form and has nice
properties; see Section \ref{Section-skein} for detailed descriptions of the skein
elements $Q_{\lambda}$ and $P_\mu$. 

Given $n\in \mathbb{Z}$, we introduce the notations $\{n\}=q^n-q^{-n}$ and  $[n]=\frac{\{n\}}{\{1\}}$. In particular,  we let $z=\{1\}=q-q^{-1}$ throughout this article. 

We study the
following reformulated colored HOMFLY-PT invariant, which is given by
\begin{align}
\mathcal{Z}_{\vec{\mu}}(\mathcal{L};q,a)=\mathcal{H}(\mathcal{L}\star
\otimes_{\alpha=1}^LP_{\mu^\alpha};q,a), \ \check{\mathcal{Z}}_{\vec{\mu}}(%
\mathcal{L};q,a)=\{\vec{\mu}\}\mathcal{Z}_{\vec{\mu}}(\mathcal{L};q,a),
\end{align}
where $\vec{\mu}=(\mu^1,...,\mu^L)$ with each $\mu^\alpha$ partition of a positive integer and
$\{\vec{\mu}\}=\prod_{\alpha=1}^{L}\{\mu^\alpha\}$ (cf. (\ref{formula-quminteger})). In particular, when $\vec{\mu}=((d), ..., (d))$ with $L$ partitions $(d)$ i.e., a single row shape, for $d\in \mathbb{N}$, we use the notation $\mathcal{Z}_d(\mathcal{L};q,a)$ to denote the reformulated framed
colored HOMFLY-PT invariant $\mathcal{Z}_{((d),...,(d))}(\mathcal{L};q,a)$ for simplicity. For brevity,
$\mathcal{Z}_{1}(\mathcal{L};q,a)$ is denoted as $\mathcal{Z}(\mathcal{L};q,a)$.

We introduce the Adams operator 
\begin{align} \label{definition-adams}
   \Psi_d: \mathbb{Q}(q^\pm,
a^\pm)\longrightarrow \mathbb{Q}(q^\pm,a^\pm), \ \  \Psi_d(f(q,a))=f(q^d,a^d)
\end{align}
and use the convention $``A \equiv B \mod C"$ to denote $\frac{A-B}{C} \in \mathbb{Z}[z^2, a^{\pm 1}].$ The following Hecke lifting conjecture for framed links was proposed in \cite{CLPZ23}. 

\begin{conjecture}\label{Heckeliftingconj}
Let $\mathcal{L}$ be a framed oriented link in $S^3$ with framing $\vec{\tau}=(\tau^1,\cdots,\tau^L)$ 
then for any prime number $p$, we have
\begin{align} \label{formula-heckelifting}
\check{\mathcal{Z}}_{p}(\mathcal{L};q,a)\equiv
(-1)^{(p-1)\sum_{\alpha=1}^L\tau^\alpha}
\Psi_p(\check{\mathcal{Z}}(\mathcal{L};q,a)) \mod [p]^2.
\end{align}
\end{conjecture}

Conjecture \ref{Heckeliftingconj} was verified in some cases in \cite{CLPZ23}. We should remark that the condition ``$p$ is prime" in the statement of Hecke lifting Conjecture \ref{Heckeliftingconj} is crucial.

\subsection{Main results}
Given two relative prime positive integers $d,m$, let $T_{d}^{m}$ be the $(d,m)$-torus knot of $d$ strands as shown in Figure \ref{beta}. 
\begin{figure}[!htb]
\begin{center}
\includegraphics[width=120 pt]{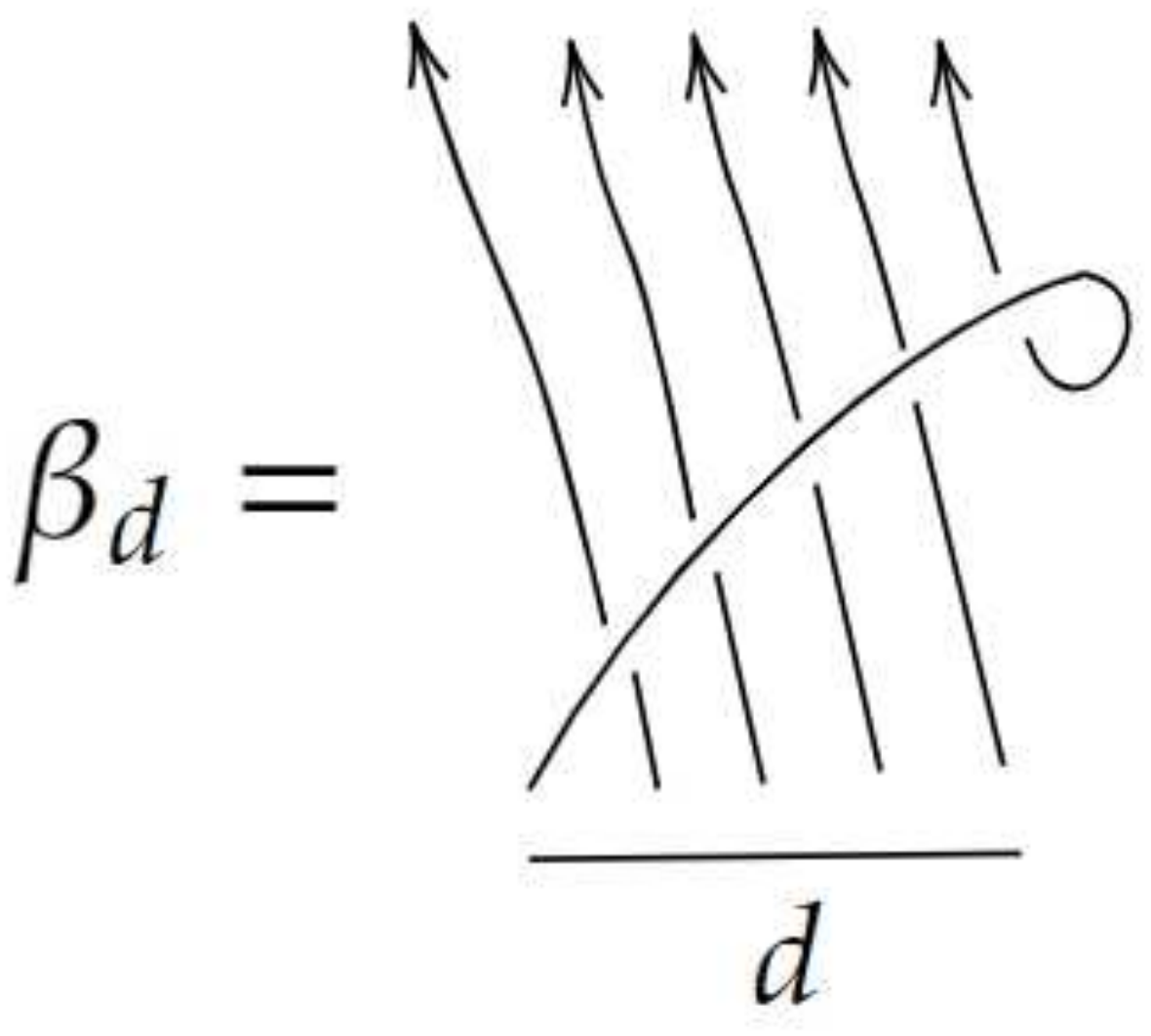}
\end{center}
\caption{$T_{d}^m$ is the closure of $(\beta_d)^m$}
\label{beta}
\end{figure}
Note that $T_{d}^m$ has the framing $dm$. Using the computations in \cite{LZ} for the torus knots, for any $p\in \mathbb{N}$, we obtain the following
\begin{align} \label{formula-Tdm}
\check{\mathcal{Z}}_{p}(T_{d}^m;q,a)&=\{p\}a^{pm}\sum_{|\mu|=pd}\frac{\{\mu\}_a}{\mathfrak{z}_\mu
\{\mu\}}\frac{\{pm\mu\}}{\{pm\}}.
\end{align}
Using the explicit expression (\ref{formula-Tdm}), we prove the following theorem.
\begin{theorem} \label{theorem-torus}
  For the torus knot $T_{d}^m$, the Hecke lifting Conjecture \ref{Heckeliftingconj} holds.   
\end{theorem}

For a framed knot $\mathcal{K}$ with a framing $\tau\in \mathbb{Z}$, we introduce the function
\begin{align}
g_p(\mathcal{K};q,a)&=\check{\mathcal{Z}}_p(\mathcal{K};q,a)-(-1)^{(p-1)\tau}\Psi_{p}(\check{\mathcal{Z}}(\mathcal{K};q,a)).  
\end{align}
Then Hecke lifting Conjecture \ref{Heckeliftingconj} for $\mathcal{K}$ is equivalent to the statement
\begin{align} \label{formula-gpintegral}
g_{p}(\mathcal{K};q,a)\in [p]^2\mathbb{Z}[z^2,a^{\pm 1}].    
\end{align}
Using the factorization property of the colored HOMFLY-PT invariants ( cf.  Theorem 1 in \cite{Morton}), the function $g_p(\mathcal{K};q,a)$ always has a factor $(a-a^{-1})$. In other words, we have the following statement which is stronger than (\ref{formula-gpintegral}): 
\begin{align} \label{formula-gp}
 g_{p}(\mathcal{K};q,a)\in (a-a^{-1})[p]^2\mathbb{Z}[z^2,a^{\pm 1}].   
\end{align}
Therefore, we have
\begin{align}
g_{p}(\mathcal{K};q,1)=0.    
\end{align}
Furthermore, (\ref{formula-gp}) is equivalent to  
\begin{align} \label{formula-gp2}
 \frac{g_{p}(\mathcal{K};q,a)}{a-a^{-1}}\in [p]^2\mathbb{Z}[z^2,a^{\pm 1}].   
\end{align}

Based on a recent result  due to Morozov et al \cite{MPS25}, we prove the following theorem which verifies formula (\ref{formula-gp2}) in a limit form. Hence, it provides evidence for Hecke lifting Conjecture \ref{Heckeliftingconj} for any framed knots.

\begin{theorem} \label{theorem-limit}
Given a framed knot $\mathcal{K}$ with a framing $\tau\in \mathbb{Z}$ and a prime $p$, we have 
 \begin{align}
\lim_{a\rightarrow 1}\frac{g_{p}(\mathcal{K};q,a)}{a-a^{-1}}\in [p]^2\mathbb{Z}[z^2].   
\end{align}   
\end{theorem}

The remainder of the paper is organized as follows. In Section \ref{section-preliminaries}, we fix the notation frequently
used in this paper and introduce the HOMFLY-PT skein model to give the definition
of reformulated colored HOMFLY-PT invariants. Then the framed LMOV conjecture and Hecke lifting conjecture are introduced. In Section \ref{section-Heckelifting-torus}, we provide a direct proof of the Hecke lifting conjecture for torus knots, i.e. Theorem \ref{theorem-torus}. In Section 
\ref{Section-coloredAlexander}, we use a recent result of the colored Alexander polynomial and its relationship to the Hecke lifting conjecture to prove Theorem \ref{theorem-limit}.

\section{Preliminaries} \label{section-preliminaries}
\subsection{Basic notations}
We first introduce some basic notations that will be used in this paper. A partition $\lambda$ is a finite sequence of positive integers $%
(\lambda_1,\lambda_2,..)$ such that $\lambda_1\geq
\lambda_2\geq\cdots$. The length of $\lambda$ is the total number of
parts in $\lambda$ and denoted by
$l(\lambda)$. The weight of $\lambda$ is defined by $|\lambda|=%
\sum_{i=1}^{l(\lambda)}\lambda_i$. The automorphism group of
$\lambda$, denoted by Aut($\lambda$), contains all the permutations
that
permute parts of $\lambda$ by keeping it as a partition. Obviously, Aut($%
\lambda$) has the order $|\text{Aut}(\lambda)|=\prod_{i=1}^{l(\lambda)}m_i(%
\lambda)! $ where $m_i(\lambda)$ denotes the number of times that
$i$ occurs in $\lambda$. Define
$\mathfrak{z}_{\lambda}=|\text{Aut}(\lambda)|\prod_{i=1}^{\lambda}\lambda_i$.

In the following, we will use the notation $\mathcal{P}_+$ to denote
the set of all the partitions of positive integers. Let $\emptyset$ be the
partition of $0$, i.e. the empty partition. Define
$\mathcal{P}=\mathcal{P}_+\cup \{\emptyset\}$.

The power sum symmetric function of infinite variables
$\mathbf{x}=(x_1,..,x_N,..)$ is defined by
$p_{n}(\mathbf{x})=\sum_{i}x_i^n. $ Given a partition $\lambda$, we
define
$p_\lambda(\mathbf{x})=\prod_{j=1}^{l(\lambda)}p_{\lambda_j}(\mathbf{x}).
$ The Schur function $s_{\lambda}(\mathbf{x})$ is determined by the
Frobenius formula
\begin{align} \label{formula-frob}
s_\lambda(\mathbf{x})=\sum_{\mu}\frac{\chi_{\lambda}(\mu)}{\mathfrak{z}_\mu}p_\mu(\mathbf{x}),
\end{align}
where $\chi_\lambda$ is the character of the irreducible
representation of
the symmetric group $S_{|\lambda|}$ corresponding to $\lambda$, we have $%
\chi_{\lambda}(\mu)=0$ if $|\mu|\neq |\lambda|$. The orthogonality
of character formula gives
\begin{align}
\sum_\lambda\frac{\chi_\lambda(\mu)
\chi_\lambda(\nu)}{\mathfrak{z}_\mu}=\delta_{\mu \nu}.
\end{align}
Let $n\in \mathbb{N}$ and $\lambda,\mu,\nu$ denote the partitions.
We introduce the following notation:
\begin{align} \label{formula-quminteger}
\{n\}_x=x^{n}-x^{-n}, \
\{\mu\}_{x}=\prod_{i=1}^{l(\mu)}\{\mu_i\}_x.
\end{align}
In particular, let $\{n\}=\{n\}_q$ and $\{\mu\}=\{\mu\}_q$.

\subsection{HOMFLY-PT skein module} \label{Section-skein}
We follow the notation in \cite{HM}. Define the coefficient ring $\Lambda=%
\mathbb{Z}[q^{\pm 1}, a^{\pm 1} ]$ with the elements $q^{k}-q^{-k}$
admitted as denominators for $k\geq 1$. Let $F$ be a planar surface,
the framed HOMFLY-PT skein $\mathcal{S}(F)$ of $F$ is the
$\Lambda$-linear combination of the orientated tangles in $F$,
modulo the two local relations as shown in Figure \ref{figure-local}, where
$z=q-q^{-1}$.
\begin{figure}[!htb]
\begin{center}
\includegraphics[width=180 pt]{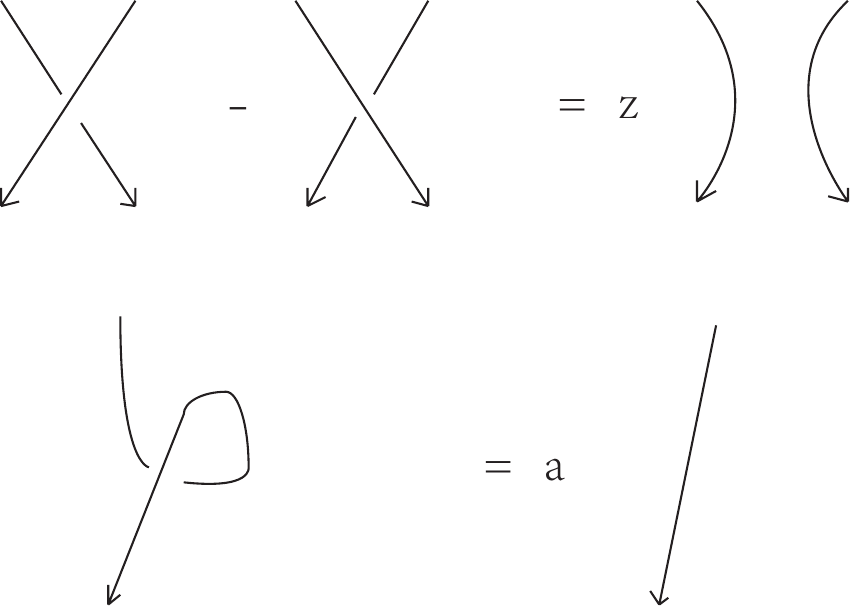}
\end{center}
\caption{Local relations}
\label{figure-local}
\end{figure}
It is easy to see that the removal of an unknot $U$ is equivalent to
time a scalar $s=\frac{a-a^{-1}}{q-q^{-1}}$, i.e., we have the
relation shown in Figure \ref{figure-removal}.
\begin{figure}[!htb]
\begin{center}
\includegraphics[width=100 pt]{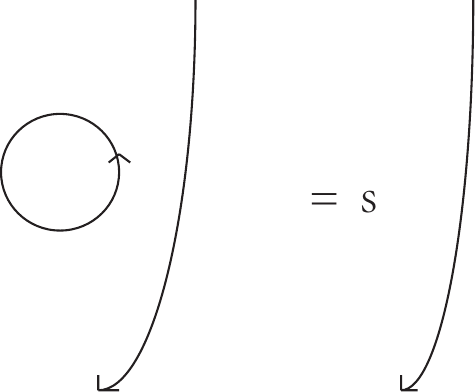}
\end{center}
\caption{Removal of an unknot}
\label{figure-removal}
\end{figure}

When $F=\mathbb{R}^2$, it is easy to follow that every element in $\mathcal{S%
}(F)$ can be represented as a scalar in $\Lambda$. For a link
$\mathcal{L}$
with a diagram $D_{\mathcal{L}}$, the resulting scalar $\langle D_{\mathcal{L%
}} \rangle \in \Lambda$ is the framed HOMFLY-PT polynomial $%
\mathcal{H}(\mathcal{L};q,a)$ of the link $\mathcal{L}$, i.e. 
\begin{align} \label{formula-H}
\mathcal{H}(
\mathcal{L};q,a)=\langle D_\mathcal{L}\rangle. 
\end{align}

We use the convention $%
\langle \ \rangle=1$ for the empty diagram; hence
\begin{align}
   \mathcal{H}(U;q,a)=\frac{
a-a^{-1}}{q-q^{-1}}. 
\end{align}

The classical HOMFLY-PT polynomial of a link $\mathcal{L}$ is given
by
\begin{align}
P(\mathcal{L};q,a)=\frac{a^{-w(\mathcal{L})}\mathcal{H}(\mathcal{L};q,a)}{%
\mathcal{H}(U;q,a)},
\end{align}
where $w(\mathcal{L})$ denotes the writhe number of the link diagram
$D_\mathcal{L}$.

When $F=S^1\times I$ is the annulus, where $I=[0,1]$, we denote $\mathcal{C}=\mathcal{S}%
(S^1\times I)$. $\mathcal{C}$ is a commutative algebra with the
product
induced by placing the annulus one outside other. As an algebra,
$\mathcal{C}$ is freely generated by the set $\{A_m: m\in
\mathbb{Z}\}$, $A_m$ for $m\neq 0$ is the closure of the braid
$\sigma_{|m|-1}\cdots \sigma_2\sigma_1$, the orientation of the
curve around the annulus is counter-clockwise for positive
$m$ and clockwise for negative $m$.  $%
A_0 $ is the empty diagram \cite{Turaev2}. It follows that
$\mathcal{C}$ contains two subalgebras $\mathcal{C}_{+}$ and
$\mathcal{C}_{-}$ that are generated by 
\begin{align}
  \{A_m: m\in \mathbb{Z},
m\geq 0\} \ \text{and} \ \{A_m:m\in \mathbb{Z},
m\leq 0\}.  
\end{align}
$\mathcal{C}_+$ can
be viewed as an algebra of symmetric functions.

 let $A_{i,j}$ be the closure of the braid $\sigma_{i+j}%
\sigma_{i+j-1}\cdots \sigma_{j+1}\sigma_{j}^{-1}\cdots
\sigma_1^{-1}$. We
define the element $X_m$ in $\mathcal{C}_+$ as 
\begin{align}
    X_m=
\sum_{i=0}^{m-1}A_{i,m-1-i}.
\end{align}
Then
$
P_m=\frac{\{1\}}{\{m\}}X_m
$ is the correspondence of power sum symmetric function $p_m(\mathbf{x})$ in $\mathcal{C}_+$.

Denoted by $Q_{\lambda}$ the closures of idempotent elements
$e_{\lambda}$
in Hecke algebra $H_m$ \cite{Ai}. \cite{L} showed that $Q_\lambda$ represent the Schur functions in the interpretation as
symmetric
functions. Hence $\{Q_{\lambda}\}_{\lambda\vdash m}$ forms the basis of $%
\mathcal{C}_m$. Furthermore, the Frobenius formula
(\ref{formula-frob}) gives
\begin{align} \label{formula-frobenisQ}
Q_\lambda=\sum_{\mu}\frac{\chi_{\lambda}(\mu)}{\mathfrak{z}_{\mu}}P_{\mu},
\end{align}
where $P_{\mu}=\prod_{i=1}^{l(\mu)}P_{\mu_i}$.

\subsection{Colored HOMFLY-PT invariants} \label{section-coloredH}
Let $\mathcal{L}$ be a framed oriented link with $L$ components:
$\mathcal{K}_1,..,\mathcal{K}_L$. For diagrams $Q_1,..,Q_L$ in the
skein model of annulus with the
positive oriented core $\mathcal{C}_+$, a link $\mathcal{L}$ decorated with $%
Q_1,...,Q_L$, denoted by $\mathcal{L}\star \otimes_{i=1}^{L} Q_i $,
is constructed by replacing every annulus $\mathcal{K}_{i}$ by the
annulus with the diagram $Q_i$ such that the orientations of the
cores match. Each $Q_i$ has a small backboard neighborhood in the
annulus which makes the decorated link
$\mathcal{L}\star\otimes_{i=1}^{L}Q_i$ a framed link (see
Figure \ref{figure-decrate} for a framed trefoil $\mathcal{K}$ decorated with skein
element $\mathcal{Q}$).

\begin{figure}[!htb]
\begin{align*}
\mathcal{K} \qquad\qquad\qquad\quad \mathcal{Q}
\qquad\qquad\qquad\quad
\mathcal{K}\star \mathcal{Q} \\
\includegraphics[width=50 pt]{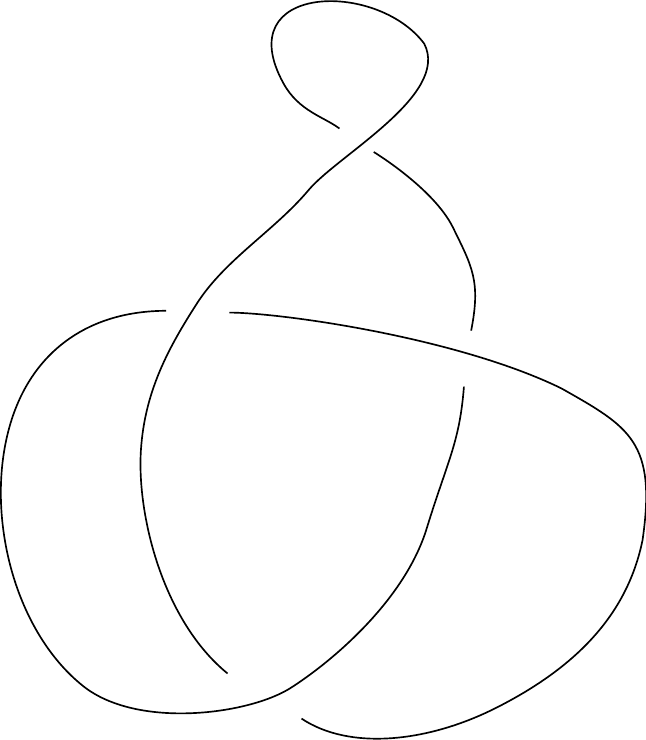}\qquad\qquad
\includegraphics[width=50
pt]{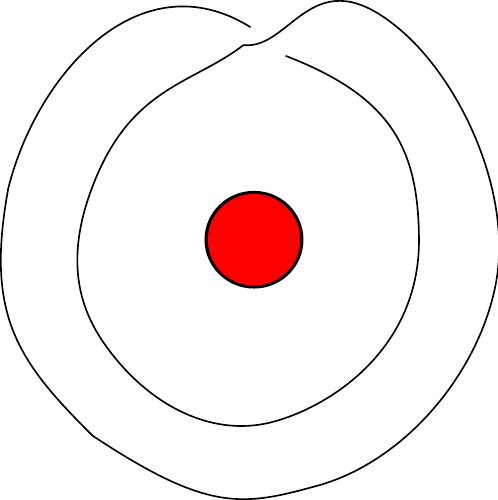} \qquad\quad
\includegraphics[width=50
pt]{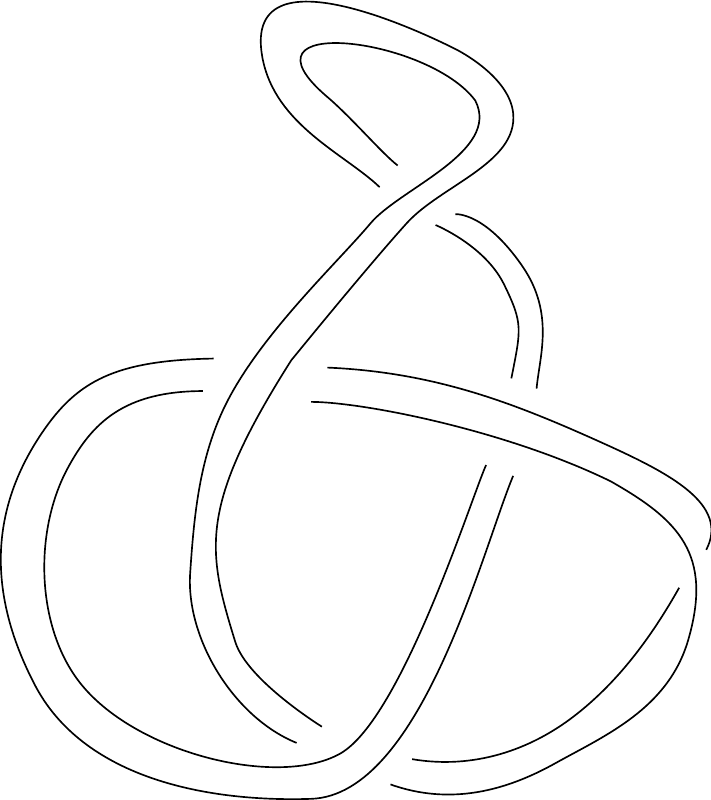}
\end{align*}%
\caption{$\mathcal{K}$ decorated by $\mathcal{Q}$} \label{figure-decrate}
\end{figure}

In particular, when $Q_{\lambda^\alpha}\in \mathcal{C}_{d_\alpha}$, where $%
\lambda^\alpha$ is a partition of a positive integer $d_\alpha$, for $%
\alpha=1,..,L$. Let $\vec{\lambda}=(\lambda^1,...,\lambda^L)$ and
 $Q_{\vec{\lambda}}=\otimes_{\alpha=1}^L
Q_{\lambda^\alpha}$.

\begin{definition}
   The framed colored HOMFLY-PT
invariant of $\mathcal{L}$ is defined as 
\begin{align}
    \mathcal{H} (\mathcal{L}\star Q_{\vec{\lambda}};q,a),
\end{align}
which is the framed HOMFLY-PT invariant of the new link $\mathcal{L}\star Q_{\vec{\lambda}}$.  Moreover, given $\vec{\mu}\in \mathcal{P}^L$,  taking $P_{\vec{\mu
}}=\otimes_{\alpha=1}^L P_{\mu^\alpha}$, we introduce the reformulated framed colored
HOMFLY-PT invariants as follows:
\begin{align} \label{formula-reformulatedinvariant}
\mathcal{Z}_{\vec{\mu}}(\mathcal{L};q,a)&=\mathcal{H}(\mathcal{L}\star
P_{\vec{\mu}};q,a), \\ \nonumber
\check{\mathcal{Z}}_{\vec{\mu}}(\mathcal{L};q,a)&=\{\vec{\mu}\}\mathcal{Z}_{\vec{\mu}}(\mathcal{L};q,a)=\{\vec{\mu}\}\mathcal{H}(\mathcal{L}\star
P_{\vec{\mu}};q,a).
\end{align}
\end{definition}

\subsection{Framed LMOV conejcture for framed knots}

 Given a framed oriented knot $\mathcal{K}$ in $S^3$ with framing $\tau\in \mathbb{Z}$,  we define the
framed Chern-Simons partition function
\begin{align} \label{framedCS}
\mathcal{Z}_{CS}(\mathcal{K};q,a)&=\sum_{\lambda\in
\mathcal{P}}(-1)^{\tau|\lambda|}\mathcal{H}(\mathcal{K}\star
Q_{\lambda};q,a)s_{\lambda}(x)\\\nonumber
&=1+\sum_{\mu\in \mathcal{P}_+}(-1)^{\tau|\mu|}
\frac{\mathcal{Z}_{\mu}(\mathcal{K};q,a)}{\mathfrak{z}_{\mu}}p_{\mu}(x),
\end{align}
and the framed free energy
\begin{align} \label{formula-CSFree}
\mathcal{F}_{CS}(\mathcal{K};q,a)=\log
\mathcal{Z}_{CS}(\mathcal{K};q,a)=\sum_{\mu\in \mathcal{P}_+}\mathcal{F}_{\mu}(\mathcal{K};q,a)p_{\mu}(x),
\end{align}

Let $f_{\lambda}(\mathcal{K};q,a)$ be a function
determined by the following formula:
\begin{align} \label{formula-CSf}
\mathcal{F}_{CS}(\mathcal{K};q,a)=\sum_{d=1}^\infty\frac{1}{d}\sum_{\lambda\in
\mathcal{P}_+}f_{\lambda}(\mathcal{K};q^d,a^d)s_{\lambda}(x^d).
\end{align}
We define
\begin{align*}
\hat{f}_{\mu}(\mathcal{K};q,a)=\sum_{\lambda}f_{\lambda}(\mathcal{K};q,a)M_{\lambda\mu}(q)^{-1},
\end{align*}
where
\begin{align*}
M_{\lambda\mu}(q)=\sum_{\nu}\frac{\chi_{\lambda}(\nu)\chi_{\mu}(\nu)}{\mathfrak{z}_{\nu}}\prod_{j=1}^{l(\nu)}(q^{\nu_{j}}-q^{-\nu_{j}}).
\end{align*}

Then the framed LMOV conjecture for framed knots \cite{MV,CLPZ23} states that:
\begin{conjecture} \label{conjecture-LMOV}
Let $\mathcal{K}$ be a framed oriented knot in $S^3$ with framing $\tau\in \mathbb{Z}$, for a fixed $\mu\in \mathcal{P}_+$, we have
\begin{align}
\hat{f}_{\mu}(\mathcal{K};q,a) \in
z^{-2}\mathbb{Z}[z^2,a^{\pm 1}],
\end{align}
where $z=q-q^{-1}$.
\end{conjecture}
For the case of a framed unknot,  the Conjecture \ref{conjecture-LMOV}  has been studied carefully
in \cite{LZ19,Zhu19}.

\subsection{Hecke lifting conjecture}
In \cite{CLPZ23}, we have shown that if the framed LMOV Conjecture \ref{conjecture-LMOV} holds, then it implies that the Hecke lifting Conjecture \ref{Heckeliftingconj} for framed knots holds, i.e.  
given a framed oriented knot $\mathcal{K}$ in $S^3$ with framing $\tau\in \mathbb{Z}$, for any prime number $p$, we have
\begin{align} \label{formula-heckelifting}
\check{\mathcal{Z}}_{p}(\mathcal{K};q,a)\equiv
(-1)^{(p-1)\tau}
\Psi_p(\check{\mathcal{Z}}(\mathcal{K};q,a)) \mod [p]^2.
\end{align}

In order to prove formula (\ref{formula-heckelifting}), we proposed two types of congruence skein relations for the reformulated colored HOMFLY-PT invariant in \cite{CLPZ23}. We found that these congruence skein relations imply that (\ref{formula-heckelifting}) also holds for framed links. That is the motivation for us to propose the Hecke lifting Conjecture \ref{Heckeliftingconj} for general framed links. However, although the framed LMOV conjecture can be formulated for general framed links, it cannot imply the Hecke lifting Conjecture \ref{Heckeliftingconj} for framed links. Hence, Hecke lifting Conjecture \ref{Heckeliftingconj} for framed links is of independent interests.

\section{Hecke lifting conjecture for torus knots} \label{section-Heckelifting-torus}
Given two relative prime positive integers $d$ and $m$. Let $T_{d}^{m}$ be the $(d,m)$-torus knot of $d$ strands with the canonical framing $dm$ as shown in Figure \ref{beta}. We define the fractional twist map 
\begin{align}
\mathfrak{f}^{\frac{m}{d}}: \mathcal{C}_+\rightarrow \mathcal{C}_+  
\end{align}
by 
\begin{align}
\mathfrak{f}^{\frac{m}{d}}(Q_\lambda)=\left(q^{\kappa_{\lambda}}a^{|\lambda|}\right)^\frac{m}{d}Q_{\lambda}.    
\end{align}
We need the following lemma. 
\begin{lemma}[cf. Lemma 4.2 in \cite{CLPZ23}] \label{Phi-lemma}
Given $\mu,\nu\in \mathcal{P}_+$, we define the following function
\begin{align*}
\phi_{\mu,\nu}(x)=\sum_{\lambda}\chi_\lambda(\mu)\chi_\lambda(\nu)x^{\kappa_\lambda},
\end{align*}
and we have
$$\phi_{(d),\mu}(x)=\frac{\prod_{i=1}^{l(\mu)}(x^{d\mu_i}-x^{-d\mu_i})}{x^d-x^{-d}}.$$
\end{lemma}

\begin{proposition} \label{prop-fmd}
For any $d,m,k\in \mathbb{N}$, we have
   \begin{align}
 \mathfrak{f}^{\frac{m}{d}}(P_{kd})=a^{km}    \sum_{|\mu|=kd}\frac{P_{\mu}}{\mathfrak{z}_\mu}\frac{\{km\mu\}}{\{km\}}.
\end{align} 
\end{proposition}
\begin{proof}
By Lemma \ref{Phi-lemma}, we have
    \begin{align}  
\mathfrak{f}^\frac{m}{d}(P_{kd})&=\sum_{|\lambda|=kd}\chi_{\lambda}((kd))\mathfrak{f}
^\frac{m}{d}(Q_\lambda) \\
&=\sum_{|\lambda|=kd}\chi_\lambda((kd))\left(q^{\kappa_\lambda}a^{kd}\right)^{\frac{m}{d}}\sum_{|\mu|=kd}\frac{\chi_\lambda(\mu)}{\mathfrak{z}_\mu}P_\mu  \notag \\
&=a^{km}\sum_{|\mu|=kd}\frac{P_\mu}{\mathfrak{z}_\mu}\phi_{(kd),\mu}(q^\frac{m}{d})\\\nonumber
&=a^{km}\sum_{|\mu|=kd}\frac{P_{\mu}}{\mathfrak{z}_\mu}\frac{\{km\mu\}}{\{km\}}.
\end{align}
\end{proof}

Note that formula (5) in \cite{MM} gives
\begin{align}
   T_{d}^m \star P_{k}=\mathfrak{f}^{\frac{m}{d}}(P_{kd}).  
\end{align}
Then Proposition \ref{prop-fmd} implies that 
\begin{align}
T_{d}^m \star
P_{p}=a^{pm}\sum_{|\mu|=pd}\frac{P_{\mu}}{\mathfrak{z}_\mu}\frac{\{pm\mu\}}{\{pm\}}.
\end{align}
By the definition of reformulated invariants (\ref{formula-reformulatedinvariant}), we obtain 
\begin{align}
\check{\mathcal{Z}}_{p}(T_{d}^m;q,a)&=\{p\}a^{pm}\sum_{|\mu|=pd}\frac{\{\mu\}_a}{\mathfrak{z}_\mu
\{\mu\}}\frac{\{pm\mu\}}{\{pm\}}, 
\end{align}
and
\begin{align}
\check{\mathcal{Z}}(T_{d}^m;q,a)&=\{1\}a^{m}\sum_{|\mu|=d}\frac{\{\mu\}_a}{\mathfrak{z}_\mu
\{\mu\}}\frac{\{m\mu\}}{\{m\}}.
\end{align}

Applying the Adams operator (\ref{definition-adams}), we obtain
\begin{align}
\Psi_p(\check{\mathcal{Z}}(T_{d}^m;q,a))=\{p\}a^{pm}\sum_{|\mu|=d}\frac{\{p\mu\}_a}{\mathfrak{z}_\mu
\{p\mu\}}\frac{\{pm\mu\}}{\{pm\}}.
\end{align}
The Hecke lifting conjecture for $T_{d}^m$ is equivalent to the statement:
\begin{align}
    \{p\}a^{pm}\left(\sum_{|\mu|=pd}\frac{\{pm\mu\}}{\mathfrak{z}_{\mu}\{pm\}}\frac{\{\mu\}_a}{\{\mu\}}
    -(-1)^{(p-1)dm}\sum_{|\nu|=d}\frac{\{pm\nu\}}{\mathfrak{z}_\nu \{pm\}}\frac{\{p\nu\}_a}{\{p\nu\}}\right)\in [p]^2\mathbb{Z}[z^2,a^{\pm 1}].
\end{align}
Set
\begin{align} \label{formula-Fpdm}
 F_{p,d,m}(q,a)= \frac{\{1\}^2}{\{p\}}a^{pm}\left(\sum_{|\mu|=pd}\frac{\{pm\mu\}}{\mathfrak{z}_{\mu}\{pm\}}\frac{\{\mu\}_a}{\{\mu\}}
    -(-1)^{(p-1)dm}\sum_{|\nu|=d}\frac{\{pm\nu\}}{\mathfrak{z}_\nu \{pm\}}\frac{\{p\nu\}_a}{\{p\nu\}}\right),   
\end{align}
so in order to prove Theorem \ref{theorem-torus}, we only need to show that
\begin{align}
   F_{p,d,m}(q,a)\in \mathbb{Z}[z^2,a^{\pm 1}].
\end{align}

We introduce the function
\begin{align}
    Q_{n}(x^k)=\frac{x^{nk}-x^{-nk}}{x^k-x^{-k}}=(x^{k})^{n-1}+(x^{k})^{n-3}+\cdots+(x^{k})^{-(n-3)}+(x^k)^{-(n-1)},
\end{align}
and use it to first prove the following two lemmas.  
\begin{lemma} \label{lemma-1}
Given two relative prime positive integers $d$ and $m$.  For any partition $\nu$ with $|\nu|=d$ and $l(\nu)=l$, we have 
\begin{align}
\frac{\{p^2m\nu\}-(-1)^{(p-1)dm}p^l\{pm\nu\}}{[pm][p]\{p\nu\}p^{l}}\in \mathbb{Q}[z^2].
\end{align}
\end{lemma}
\begin{proof}
Let 
\begin{align}
  f_{p,m,\nu}(x)=\frac{\prod_{i=1}^{l}Q_{pm}(x^{p\nu_i})-(-1)^{(p-1)dm}p^l\prod_{i=1}^lQ_{m}(x^{p\nu_i})}{Q_{pm}(x)Q_p(x)},
\end{align}
then
\begin{align}
   \frac{\{p^2m\nu\}-(-1)^{(p-1)dm}p^l\{pm\nu\}}{[pm][p]\{p\nu\}p^{l}}=f_{p,m,\nu}(q). 
\end{align}
So we only need to prove
\begin{align}
    f_{p,m,\nu}(x)\in \mathbb{Q}[(x-x^{-1})^2].
\end{align}

Let 
\begin{align}
  g_{p,m,\nu}(x)=\prod_{i=1}^{l}Q_{pm}(x^{p\nu_i})-(-1)^{(p-1)dm}p^l\prod_{i=1}^lQ_{m}(x^{p\nu_i}). 
\end{align}

(i) For the case $p=2$, we need to prove that
\begin{align}
    \frac{g_{2,m,\nu}(x)}{Q_{2m}(x)Q_2(x)}\in \mathbb{Q}[(x-x^{-1})^2].
\end{align}

Suppose $\alpha$ is a root of $Q_2(x)$, i.e. $\alpha^2=-1$, it is easy to see $Q_{2m}(\alpha)=0$,  $\frac{dQ_{2m}(x)}{dx}|_{x=\alpha}\neq 0$. 
So $\alpha$ is a double root of $Q_{2m}(x)Q_2(x)$. 
We have 
\begin{align}
    \prod_{i=1}^lQ_{2m}(\alpha^{2\nu_i})=\prod_{i=1}^l(2m(-1)^{\nu_i})=(2m)^l(-1)^d.
\end{align}
and 
\begin{align}
    \prod_{i=1}^lQ_{m}(\alpha^{2\nu_i})=m^l((-1)^d)^{m-1}. 
\end{align}
Then 
\begin{align}
  g_{2,m,\nu}(\alpha)&=\prod_{i=1}^lQ_{2m}(\alpha^{2\nu_i})-(-1)^{dm}2^l\prod_{i=1}^lQ_{m}(\alpha^{2\nu_i})\\\nonumber
  &=(2m)^l(-1)^d-(-1)^{dm}2^lm^l((-1)^d)^{m-1}=0.
\end{align}
Moreover, we compute that, for $1\leq i\leq l$ 
\begin{align}
    \frac{dQ_{2m}(x^{2\nu_i})}{dx}|_{x=\alpha}=0, \ 
    \frac{dQ_{m}(x^{2\nu_i})}{dx}|_{x=\alpha}=0. 
\end{align}
It implies that 
\begin{align}
 \frac{dg_{2,m,\nu}(x)}{dx}|_{x=\alpha}=0.    
\end{align}

Now, suppose $\beta$ is a root of $Q_{2m}(x)$, but not a root of $Q_2(x)$. Then, we have 
\begin{align}
    \beta^{2m}=\pm 1, \ \beta^2\neq \pm 1.
\end{align}
$\beta$ is a simple root of the polynomial $Q_{2m}(x)Q_2(x)$. 
We can write $\beta=e^{\frac{s\pi \sqrt{-1}}{2m_0}}$ with $(s,2m_0)=1$ and $m_0|m$ but $m_0\neq \pm 1$. 
We observe that for $i=1,...,l$, there is at least one $\nu_i$ such that $m_0\nmid \nu_i$. Since if for all $i$, we have
$m_0\mid \nu_i$, then we must have $m_0|d$ which contradicts the condition $(d,m)=1$.  
Since
\begin{align}
 Q_{2m}(\beta^{2\nu_i})=\frac{(\beta^{2\nu_i})^{2m}-(\beta^{2\nu_i})^{-2m}}{\beta^{2\nu_i}-\beta^{-2\nu_i}},   
\end{align}
if $m_0\nmid \nu_i$, then $\beta^{2\nu_i}-\beta^{-2\nu_i}\neq 0$ and $(\beta^{2\nu_i})^{2m}-(\beta^{2\nu_i})^{-2m}=0$, 
hence $Q_{2m}(\beta^{2\nu_i})=0$.  
Therefore, we have $\prod_{i=1}^lQ_{2m}(\beta^{2\nu_i})=0$. Similarly, we also have $\prod_{i=1}^lQ_{m}(\beta^{2\nu_i})=0$, i.e. 
\begin{align}
    g_{2,m,\nu}(\beta)=0. 
\end{align}

Since all the possible roots of $Q_{2m}(x)Q_m(x)$ are $\alpha,\beta$, from the above analysis we obtain 
$f_{2,m,\nu}(x)$ is a polynomial of $x$. Moreover, it is easy to see that 
\begin{align}
    f_{2,m,\nu}(x)=f_{2,m,\nu}(x^{-1}) \ \text{and} \ f_{2,m,\nu}(-x)=f_{2,m,\nu}(x). 
\end{align}
We obtain 
\begin{align}
    f_{2,m,\nu}(x)\in \mathbb{Q}[(x-x^{-1})^2]. 
\end{align}

(ii) For the generic odd prime $p$,  

\begin{align}
  g_{p,m,\nu}(x)=\prod_{i=1}^{l}Q_{pm}(x^{p\nu_i})-p^l\prod_{i=1}^lQ_{m}(x^{p\nu_i}). 
\end{align}

We consider the roots of $Q_{pm}(x)Q_p(x)$.  Suppose $\alpha$ is a root of  $Q_p(x)$, we write $\alpha=e^{\frac{s\pi \sqrt{-1}}{p}}$. 
Then, we obtain $Q_{pm}(\alpha)=0$ and $\frac{d Q_{pm}(x)}{dx}|_{x=\alpha}\neq 0$. So $\alpha$ is a double root of  $Q_{pm}(x)Q_p(x)$.

If $m$ is even, we obtain 
\begin{align}
     Q_{pm}(y)&=y^{pm-1}+\cdots+y+y^{-1}+\cdots +y^{-(pm-1)}, \\\nonumber
     Q_{m}(y)&=y^{m-1}+\cdots+y+y^{-1}+\cdots+y^{-(m-1)}. 
\end{align}
Then 
\begin{align}
 Q_{pm}(\alpha^{p\nu_i})=Q_{pm}((-1)^{s\nu_i})=(-1)^{s\nu_i}pm, \ Q_{m}(\alpha^{p\nu_i})=Q_{m}((-1)^{s\nu_i})=(-1)^{s\nu_i}m.
\end{align}
So we obtain 
\begin{align}
g_{p,m,\nu}(\alpha)=(-1)^{sd}(pm)^l-p^l(-1)^{sd}m^l=0.
\end{align}

If $m$ is odd, then $pm$ is odd since $p$ is an odd prime. In this case 
\begin{align}
     Q_{pm}(y)&=y^{pm-1}+\cdots+y^2+1+y^{-2}+\cdots +y^{-(pm-1)}, \\\nonumber
     Q_{m}(y)&=y^{m-1}+\cdots+y^2+1+y^{-2}+\cdots +y^{-(m-1)}. 
\end{align}
Then
\begin{align}
 Q_{pm}(\alpha^{p\nu_i})=Q_{pm}((-1)^{s\nu_i})=pm, \ Q_{m}(\alpha^{p\nu_i})=Q_{m}((-1)^{s\nu_i})=m.
\end{align}
So we obtain 
\begin{align}
g_{p,m,\nu}(\alpha)=(pm)^l-p^lm^l=0.
\end{align}

Moreover, we compute that, for $1\leq i\leq l$ 
\begin{align}
    \frac{dQ_{pm}(x^{p\nu_i})}{dx}|_{x=\alpha}=0, \ 
    \frac{dQ_{m}(x^{p\nu_i})}{dx}|_{x=\alpha}=0. 
\end{align}
It implies that 
\begin{align}
 \frac{dg_{p,m,\nu}(x)}{dx}|_{x=\alpha}=0.    
\end{align}

Now, suppose $\beta$ is a root of $Q_{pm}(x)$, but not a root of $Q_p(x)$. Then, we have 
\begin{align}
    \beta^{pm}=\pm 1, \ \beta^p\neq \pm 1.
\end{align}
$\beta$ is a root of one order of the polynomial $Q_{pm}(x)Q_p(x)$. 
We can write $\beta=e^{\frac{s\pi \sqrt{-1}}{pm_0}}$ with $(s,pm_0)=1$ and $m_0|m$ but $m_0\neq \pm 1$. 
We observe that for $i=1,...,l$, there is at least one $\nu_i$ such that $m_0\nmid \nu_i$. Since if for all $i$, we have
$m_0\mid \nu_i$, then we must have $m_0|d$ which contradicts the condition $(d,m)=1$.  
Since
\begin{align}
 Q_{pm}(\beta^{p\nu_i})=\frac{(\beta^{p\nu_i})^{pm}-(\beta^{p\nu_i})^{-pm}}{\beta^{p\nu_i}-\beta^{-p\nu_i}},   
\end{align}
if $m_0\nmid \nu_i$, then $\beta^{p\nu_i}-\beta^{-p\nu_i}\neq 0$ and $(\beta^{p\nu_i})^{pm}-(\beta^{p\nu_i})^{-pm}=0$, 
hence $Q_{pm}(\beta^{p\nu_i})=0$.  
Therefore, we have $\prod_{i=1}^lQ_{pm}(\beta^{p\nu_i})=0$. Similarly, we also have $\prod_{i=1}^lQ_{m}(\beta^{p\nu_i})=0$, i.e. 
\begin{align}
    g_{p,m,\nu}(\beta)=0. 
\end{align}

Since all the possible roots of $Q_{pm}(x)Q_m(x)$ are $\alpha,\beta$, from the above analysis we obtain 
$f_{p,m,\nu}(x)$ is a polynomial of $x$. Moreover, it is easy to see that 
\begin{align}
    f_{p,m,\nu}(x)=f_{p,m,\nu}(x^{-1}) \ \text{and} \ f_{p,m,\nu}(-x)=f_{p,m,\nu}(x). 
\end{align}
We obtain 
\begin{align}
    f_{p,m,\nu}(x)\in \mathbb{Q}[(x-x^{-1})^2]. 
\end{align}
\end{proof}

\begin{lemma} \label{lemma-2}
Given two relative prime positive integers $d$ and $m$. For any partition $\mu$ with $|\mu|=pd$ and $p\nmid \mu$, we have 
\begin{align}
    \frac{\{pm\mu\}}{[p][pm]\{\mu\}}\in \mathbb{Q}[z^2]. 
\end{align}
\end{lemma}
\begin{proof}
Let
\begin{align}
    f_{p,m,\mu}(x)=\frac{\prod_{i=1}^lQ_{pm}(x^{\mu_i})}{Q_{pm}(x)Q_{p}(x)},
\end{align}
then 
\begin{align}
  \frac{\{pm\mu\}}{[p][pm]\{\mu\}}=f_{p,m,\mu}(q).
\end{align}
In the following, we will prove that 
\begin{align}
f_{p,m,\mu}(q)\mathbb{Q}[(x-x^{-1})^2].    
\end{align}

Suppose $\alpha=e^{\frac{s\pi \sqrt{-1}}{p}}$ is a root of $Q_{p}(x)$, then $\alpha$ is a double root of $Q_{p}(x)Q_{pm}(x)$.

By conditions $|\mu|=pd$ and $p\nmid \mu$, we note that the length  $l(\mu)$ of the partition $\mu$ should $l(\mu)\geq 2$, and there are at least two parts, say $\mu_i$ and $\mu_j$, that satisfy 
\begin{align}
    p\nmid \mu_i, \ p\nmid \mu_j. 
\end{align}
Therefore, we obtain 
\begin{align}
    \alpha^{pm\mu_i}-\alpha^{-pm\mu_i}=0, \  \alpha^{pm\mu_j}-\alpha^{-pm\mu_j}=0, 
\end{align}
but 
\begin{align}
    \alpha^{\mu_i}-\alpha^{-\mu_i}\neq 0, \  \alpha^{\mu_j}-\alpha^{-\mu_j}\neq 0. 
\end{align}
Hence, 
\begin{align}
    Q_{pm }(\alpha^{\mu_i})=Q_{pm}(\alpha^{\mu_j})=0.
\end{align}

Let $\beta=e^{\frac{s\pi \sqrt{-1}}{pm_0}}$ with $(s,pm_0)=1$ and $m_0|m$ but $m_0\neq \pm 1$. $\beta$ is a root of $Q_{pm}(x)$ but  not a root of $Q_p(x)$. So, it is a single root of $Q_{pm}(x)Q_p(x)$.  

If $m_0=p$, by the above analysis we know that there is $\mu_i$ such that $p\nmid \mu_i$, hence $\beta^{\mu_i}-\beta^{-\mu_i}\neq 0$. Clearly, $\beta^{pm\mu_i}-\beta^{-pm\mu_i}=0$. Hence $Q_{pm}(\beta^{\mu_i})=0$. 

If $m_0\neq p$,  using the condition $(d,m)=1$, we know that there is at least one part, say $\mu_k$, such that $m_0\nmid \mu_k$. If not, then we have $m_0|pd$, and hence $m_0|d$, which contradicts $(d,m)=0$.  
Hence $\beta^{\mu_k}-\beta^{-\mu_k}\neq 0$. Clearly, $\beta^{pm\mu_k}-\beta^{-pm\mu_k}=0$. Hence $Q_{pm}(\beta^{\mu_k})=0$.

Since all the possible roots of $Q_{pm}(x)Q_m(x)$ are $\alpha,\beta$, from the above analysis we obtain 
$f_{p,m,\mu}(x)$ is a polynomial of $x$. It is easy to see that 
\begin{align}
    f_{p,m,\mu}(x)=f_{p,m,\mu}(x^{-1}) \ \text{and} \ f_{p,m,\mu}(-x)=f_{p,m,\mu}(x). 
\end{align}
Finally, we obtain 
\begin{align}
    f_{p,m,\mu}(x)\in \mathbb{Q}[(x-x^{-1})^2]. 
\end{align}
\end{proof}

Now, we can finish the proof of Theorem \ref{theorem-torus}. 
\begin{proof}
In order to prove that
\begin{align}
 F_{p,d,m}(q,a)\in \mathbb{Z}[z^2,a^{\pm 1}],     
\end{align}
we divide the first summation in (\ref{formula-Fpdm}) into two parts: 
\begin{align}
    \sum_{|\mu|=pd}\frac{\{pm\mu\}}{\mathfrak{z}_{\mu}[pm][p]}\frac{\{\mu\}_a}{\{\mu\}}
    &=\sum_{|\mu|=pd, p\mid \mu }\frac{\{pm\mu\}}{\mathfrak{z}_{\mu}[pm][p]}\frac{\{\mu\}_a}{\{\mu\}}+\sum_{|\mu|=pd, p\nmid \mu }\frac{\{pm\mu\}}{\mathfrak{z}_{\mu}[pm][p]}\frac{\{\mu\}_a}{\{\mu\}}\\\nonumber
    &=\sum_{|\nu|=d}\frac{\{p^2m\nu\}}{\mathfrak{z}_{\nu}p^{l(\mu)}[pm][p]}\frac{\{p\nu\}_a}{\{p\nu\}}+\sum_{|\mu|=pd, p\nmid \mu }\frac{\{pm\mu\}}{\mathfrak{z}_{\mu}[pm][p]}\frac{\{\mu\}_a}{\{\mu\}}.
\end{align}
Hence, $F_{p,d,m}(q,a)$ can be written as 
\begin{align}
     F_{p,d,m}(q,a)&=\sum_{|\nu|=d}\frac{1}{\mathfrak{z}_\nu}\left(\frac{\{p^2m\nu\}}{p^{l(\mu)}[pm][p]\{p\nu\}}-(-1)^{(p-1)dm}\frac{\{pm\nu\}}{ [pm][p]\{p\nu\}}\right)\{p\nu\}_a\\\nonumber
     &+\sum_{|\mu|=pd, p\nmid \mu }\frac{\{pm\mu\}}{\mathfrak{z}_{\mu}[pm][p]}\frac{\{\mu\}_a}{\{\mu\}}. 
\end{align}
By Lemmas \ref{lemma-1} and \ref{lemma-2}, we obtain
\begin{align}
    F_{p,d,m}(q,a)\in \mathbb{Q}[z^2,a^{\pm 1}]. 
\end{align}
The formula (\ref{formula-Fpdm}) gives 
\begin{align}
\check{\mathcal{Z}}_p(T_{d}^m;q,a)-\Psi_p(\check{\mathcal{Z}}(T_d^m;q,a))=[p]^2F_{p,d,m}(q,a).     
\end{align}
Moreover, Proposition 3.2 in \cite{CLPZ23} implies that
\begin{align}
\check{\mathcal{Z}}_p(T_{d}^m;q,a)-\Psi_p(\check{\mathcal{Z}}(T_d^m;q,a))\in \mathbb{Z}[z^2,a^{\pm 1}].    
\end{align}
Together with the fact that $[p]^2\in \mathbb{Z}[z^2]$ is a monic polynomial of $z^2$, it is easy to obtain
\begin{align}
    F_{p,d,m}(q,a)\in \mathbb{Z}[z^2,a^{\pm 1}]. 
\end{align}
\end{proof}

\section{Colored Alexander polynomial and Hecke lifting conjecture} \label{Section-coloredAlexander}
Given a framed knot $\mathcal{K}$ with framing $\tau\in \mathbb{Z}$, the (framing-independent) colored HOMFLY-PT
invariant of $\mathcal{K}$ colored by a partition $\lambda$ is defined as 
\begin{align}
   W_{\lambda}(\mathcal{K};q,a)=q^{-\kappa_{\lambda}\tau}a^{-|\lambda|\tau}\mathcal{H}(\mathcal{K}\star Q_{\lambda};q,a). 
\end{align}
In \cite{IMMM12}, Itoyama et al. considered the following limit 
\begin{align}
A_{\lambda}(\mathcal{K};q)=\lim_{a\rightarrow
1}\frac{W_{\lambda}(\mathcal{K};q,a)}{W_{\lambda}(U;q,a)}
\end{align}
which is well-defined \cite{LP1,Morton}. When $\lambda=(1)$, by its definition, we obtain
$A_{(1)}(\mathcal{K};q):=A(\mathcal{K};q)$, which is just the Alexander
polynomial for the knot $\mathcal{K}$, so we call
$A_{\lambda}(\mathcal{K};q)$ the colored Alexander polynomial of
$\mathcal{K}$. 

In \cite{BMMSS11}, Dunin-Barkowski et al. studied the properties of
$A_{\lambda}(\mathcal{K};q)$, they proposed the following conjectural formula
\begin{align} \label{formula-Alambda}
A_{\lambda}(\mathcal{K};q)=A(\mathcal{K};q^{|\lambda|})
\end{align}
when $\lambda$ is a hook partition.
In \cite{Zhu13}, it was proved that
(\ref{formula-Alambda}) holds for the torus knot. Recently, in \cite{MPS25}, Morozov, Popolitov and Sleptsov
proved the conjectural formula (\ref{formula-Alambda}) for any knot, i.e., we have
\begin{theorem}[cf. Theorem 1.1 in \cite{MPS25}] \label{theorem-coloredAlex}
When $\lambda$ is a hook partition, for any knot $\mathcal{K}$, 
\begin{align}
    A_{\lambda}(\mathcal{K};q)=A(\mathcal{K};q^{|\lambda|}).
\end{align}
\end{theorem}

Note that formula (\ref{formula-Alambda}) has a form similar to the Hecke lifting Conjecture \ref{Heckeliftingconj} for a knot. Motivating us to consider the relationship between them.  In the remainder of this section, we show that Theorem \ref{theorem-coloredAlex} actually implies Theorem \ref{theorem-limit}. Note that the essential part of this proof has been obtained in our work \cite{Zhu22}.

First, we recall that every hook partition of weight $d$ can be
presented as the form $(m+1,1,...,1)$ with $n+1$ length for some $m,
n\in \mathbb{Z}_{\geq 0}$, denoted by $(m|n)$, with $m+n+1=d$. It is
clear that $\kappa_{(m|n)}=(m-n)d$.

Let $U$ be an unknot, by formula (5.14) in \cite{LZ},
\begin{align}
W_{\lambda}(U;q,a)=\sum_{\mu}\frac{\chi_{\lambda}(\mu)}{\mathfrak{z}_{\mu}}\prod_{i=1}^{l(\mu)}\frac{\{\mu\}_a}{\{\mu\}},    
\end{align}
Using the character of the symmetry group:
\begin{align}
\chi_{(m|n)}((d))=(-1)^n,     
\end{align}
we have
\begin{align}
  W_{(m|n)}(U;q,a)=\frac{(-1)^n}{d}\frac{a^d-a^{-d}}{q^d-q^{-d}}+\sum_{l(\mu)\geq
2}\sum_{\mu}\frac{\chi_{(m|n)}(\mu)}{\mathfrak{z}_{\mu}}\prod_{j=1}^{l(\mu)}\frac{\{\mu_j\}_a}{\{\mu_j\}}.  
\end{align}

By the definition of the colored Alexander polynomial, we have
\begin{align}
A_{(m|n)}(\mathcal{K};q)&=\lim_{a\rightarrow
1}\frac{q^{-(m-n)d\tau}a^{-d\tau}\mathcal{H}(\mathcal{K}\star
Q_{(m|n)};q,a)}{W_{(m|n)}(U;q,a)}\\\nonumber
&=q^{-(m-n)d\tau}(-1)^nd(q^{d}-q^{-d})\lim_{a\rightarrow
1}\frac{\mathcal{H}(\mathcal{K}\star Q_{(m|n)};q,a)}{(a^d-a^{-d})}.
\end{align}
Then Theorem \ref{theorem-coloredAlex} implies that
\begin{align} \label{formula-lima}
\lim_{a\rightarrow 1}\frac{\mathcal{H}(\mathcal{K}\star
Q_{(m|n)};q,a)}{(a^d-a^{-d})}&=(-1)^n\frac{q^{(m-n)d\tau}}{d(q^d-q^{-d})}A_{(m|n)}(\mathcal{K};q)\\\nonumber
&=(-1)^n\frac{q^{(m-n)d\tau}}{d(q^d-q^{-d})}A(\mathcal{K};q^d).
\end{align}

Moreover, Frobenius formula (\ref{formula-frobenisQ}) follows
\begin{align}
    P_{d}=\sum_{\lambda}\chi_{\lambda}((d))Q_{\lambda}=\sum_{m+n+1=d}(-1)^nQ_{(m|n)}, 
\end{align}
since 
\makeatletter    \label{formula-chid}
\let\@@@alph\@alph
\def\@alph#1{\ifcase#1\or \or $'$\or $''$\fi}\makeatother
\begin{subnumcases}
{\chi_\lambda((d))=} (-1)^n, &\text{if $\lambda$ is a hook partition
$(m|n)$}\\\nonumber 0, &\text{otherwise.}
\end{subnumcases}
\makeatletter\let\@alph\@@@alph\makeatother
Therefore, 
\begin{align}
\lim_{a\rightarrow 1}\frac{\mathcal{H}(\mathcal{K}\star
P_d;q,a)}{a^{d}-a^{-d}}&=\lim_{a\rightarrow
1}\sum_{m+n+1=d}\frac{(-1)^n\mathcal{H}(\mathcal{K}\star
Q_{(m|n)};q,a)}{a^{d}-a^{-d}}\\\nonumber
&=\frac{A(\mathcal{K};q^d)}{d(q^d-q^{-d})}\sum_{m+n+1=d}q^{(m-n)dw(\mathcal{K})}.
\end{align}

On the other hand, for the definition of Alexander polynomial, we
have
\begin{align}
A(\mathcal{K};q)=(q-q^{-1})\lim_{a\rightarrow
1}\frac{\mathcal{H}(\mathcal{K};q,a)}{a-a^{-1}}.
\end{align}
Then
\begin{align}
\lim_{a\rightarrow
1}\Psi_{d}\left(\frac{\mathcal{H}(\mathcal{K};q,a)}{a-a^{-1}}\right)=\frac{A(\mathcal{K};q^d)}{(q^d-q^{-d})}.
\end{align}
Hence, we obtain
\begin{align} \label{formula-limH}
&\lim_{a\rightarrow 1}\left(\frac{\mathcal{H}(\mathcal{K}\star
P_d;q,a)}{(a^{d}-a^{-d})}-(-1)^{(d-1)\tau)}\Psi_{d}\left(\frac{\mathcal{H}(\mathcal{K};q,a)}{a-a^{-1}}\right)\right)\\\nonumber
&=\frac{A(\mathcal{K};q^d)}{d(q^d-q^{-d})}\left(\sum_{m+n+1=d}q^{(m-n)d\tau}-(-1)^{(d-1)\tau}\right).
\end{align}
In the following, we use the identity (\ref{formula-limH}) to prove Theorem \ref{theorem-limit}.

Given a prime $p$, by formula  (\ref{formula-limH}), we obtain
\begin{align} \label{formula-1}
&\lim_{a\rightarrow
1}\left(\frac{\mathcal{H}(\mathcal{K}\star
P_p;q,a)}{a-a^{-1}}-(-1)^{(p-1)\tau}\frac{\Psi_{p}\left(\mathcal{H}(\mathcal{K};q,a)\right)}{a-a^{-1}}\right)\\\nonumber
&=p\lim_{a\rightarrow 1}\left(\frac{\mathcal{H}(\mathcal{K}\star
P_p;q,a)}{(a^p-a^{-p})}-(-1)^{(p-1)\tau}\frac{\Psi_{p}\left(\mathcal{H}(\mathcal{K};q,a)\right)}{(a^p-a^{-p})}\right)\\\nonumber
&=\frac{A(\mathcal{K};q^p)}{(q^p-q^{-p})}\left(\sum_{m+n+1=p}q^{(m-n)d\tau}-(-1)^{(p-1)\tau}\right).
\end{align}
Using Lemma 7.8 in \cite{Zhu22} , we have
\begin{align} \label{formula-2}
\sum_{m+n+1=p}q^{(m-n)d\tau}-(-1)^{(p-1)\tau}=[p]^2\alpha_{p}^{\tau}(z),
\end{align}
where 
\begin{align} \label{formula-3}
\alpha_{p}^{\tau}(z)\in \mathbb{Z}[z^2].
\end{align}

Recall that
\begin{align}
g_p(\mathcal{K};q,a)&=\check{\mathcal{Z}}_p(\mathcal{K};q,a)-(-1)^{(p-1)\tau}\Psi_{p}(\check{\mathcal{Z}}(\mathcal{K};q,a))\\\nonumber
&=\{p\}\left(\mathcal{H}(\mathcal{K}\star
P_p;q,a)-(-1)^{(p-1)\tau}\Psi_{d}\left(\mathcal{H}(\mathcal{K};q,a)\right)\right).
\end{align}
Therefore, by (\ref{formula-1}), (\ref{formula-2}) and (\ref{formula-3}), 
\begin{align}
\lim_{a\rightarrow 1}\frac{g_p(\mathcal{K};q,a)}{(a-a^{-1})}=[p]^2A(\mathcal{K};q^p)\alpha_{p}^{\tau}(z).
\end{align}
Moreover, using the properties of the Alexander polynomial in our
notation, we have
\begin{align}
A(\mathcal{K};-q^p)=A(\mathcal{K};q^p), \
A(\mathcal{K};q^{-p})=A(\mathcal{K};q^p),
\end{align}
which implies that
\begin{align}
A(\mathcal{K};q^p)\in \mathbb{Z}[z^2].
\end{align}
So, we obtain 
\begin{align}
  \lim_{a\rightarrow 1}\frac{g_p(\mathcal{K};q,a)}{(a-a^{-1})}\in [p]^2\mathbb{Z}[z^2].  
\end{align}
The proof of Theorem \ref{theorem-limit} is complete.

\end{document}